\documentclass[12pt]{amsart}
\usepackage{amsthm}
\usepackage{inputenc}
\usepackage[all,cmtip]{xy}
\usepackage{amssymb}
\usepackage[usenames,dvipsnames]{xcolor}
\usepackage{color}

\DeclareMathOperator{\Hom}{Hom}
\DeclareMathOperator{\Rad}{Rad}
\DeclareMathOperator{\rad}{rad}
\DeclareMathOperator{\Der}{Der}
\DeclareMathOperator{\Span}{Span}
\DeclareMathOperator{\Leib}{Leib}

\theoremstyle{plain} 
\newtheorem{theorem}{Theorem}[section]
\newtheorem{corollary}[theorem]{Corollary}
\newtheorem{conjecture}[theorem]{Conjecture}

\newtheorem{proposition}[theorem]{Proposition}
\newtheorem{definition}[theorem]{Definition}
\newtheorem{remark}[theorem]{Remark}

\newtheorem{example}[theorem]{Example}
\numberwithin{equation}{section}

\newcommand{\U}{\mathbf{U}}
\newcommand{\Lie}{\mathbf{Lie}}
\newcommand{\Lb}{\mathbf{Lb}}

\begin{document}
\title[Leibniz $n$-algebras from the category $\U_n(\Lb)$]{Some theorems on Leibniz $n$-algebras from \\ the category $\U_n(\Lb)$}
\author{M.S. Kim}
\address{[Min Soo Kim] Vanderbilt University, Nashville, TN 37235}
\email{min.soo.kim@vanderbilt.edu}
\author{R. Turdibaev}
\address{[Rustam Turdibaev] Inha University in Tashkent, 100170 Tashkent, Uzbekistan}
\email{r.turdibaev@inha.uz}

\thanks{The first author was supported by the National Science Foundation, grant number 1658672.}
\begin{abstract}
We study the Leibniz $n$-algebra $\U_n(\mathfrak{L})$, whose multiplication is defined via the bracket of a Leibniz algebra $\mathfrak{L}$ as $[x_1,\dots,x_n]=[x_1,[\dots, [x_{n-2},[x_{n-1},x_n]]\dots]]$. We show that $\U_n(\mathfrak{L})$ is simple if and only if $\mathfrak{L}$ is a simple Lie algebra. An analogue of Levi's theorem for Leibniz algebras in $\U_n(\Lb)$ is established and it is proven that the Leibniz $n$-kernel of $\U_n(\mathfrak{L})$ for any semisimple Leibniz algebra $\mathfrak{L}$ is the $n$-algebra $\U_n(\mathfrak{L})$.
\end{abstract}
\maketitle

\section*{Introduction}
Leibniz algebras were introduced by A. Bloh \cite{Bloh} in 1960s as algebras satisfying the Leibniz identity:
\[[[x,y],z]=[[x,z],y]+[x,[y,z]].\]
The Leibniz identity becomes the Jacobi identity if one uses skew-symmetry and in fact, the category of Lie algebras ($\Lie$) is a full subcategory of the category of Leibniz algebras ($\Lb$). However, Leibniz algebras did not gain popularity until J.-L. Loday  \cite{LodayCyclic} rediscovered them in 1990s while lifting classical Chevalley-Eilenberg boundary map to the tensor module of a Lie algebra that produces another (Leibniz) chain complex. Leibniz algebras play an important role in Hochschild homology theory \cite{LodayCyclic} and Nambu mechanics (\cite{Nam}, \cite{DaTa}).

In 1985 Filippov \cite{Fil} defined a notion of an $n$-Lie algebra that generalizes Lie algebras by the arity of the bracket from two to $n$, preserving skew symmetry and satisfying the identity:
\[[x_1,\ldots,x_n],y_1,\ldots,y_{n-1}]=\sum^{n}_{i=1}[x_1,\ldots,x_{i-1},[x_i,y_1,\ldots,y_{n-1}],x_{i+1},\ldots,x_n].\] For $n=3$ an example of a $3$-Lie algebra appears even earlier in 1973 \cite{Nam}, where the multiplication for a triple of classical observables on the three-dimensional phase space $\mathbb{R}^3$ is given by the Jacobian and studied further in \cite{Takh}. An $n$-Lie algebra is also called a \textit{Filippov algebra} or a \textit{Nambu algebra} in numerous papers and is relevant in string and membrane theory \cite{BaLa1, BaLa2}.  

The generalization of Lie, Leibniz and $n$-Lie algebras is the so called \textit{Leibniz $n$-algebras}, introduced in \cite{CLP}. Leibniz $n$-algebras are defined with the above identity but are not skew-symmetric with the $n$-ary bracket. As in the case of Leibniz algebras $(n=2)$, the ideal generated by the $n$-ary brackets, where two elements are the same is called the \textit{Leibniz $n$-kernel} and the quotient $n$-algebra by this ideal is an $n$-Lie algebra. 

While the structure theory of Leibniz and $n$-Lie algebras has been well-developed, the same cannot be said about $n$-Leibniz algebras. Although a notion of a solvable radical and some other classical properties of the radical of Leibniz $n$-algebras are obtained in \cite{Frattini}, Levi decomposition for Leibniz $n$-algebras has not yet been established, while for $n$-Lie and Leibniz algebras they are proven by Ling \cite{Ling} and Barnes \cite{Barnes}, respectively. Even in small dimensions, there is only a partial classification of two-dimensional Leibniz $n$-algebras over complex numbers in \cite{Adashev} and \cite{5authors}. There are also some instances when results from Leibniz and $n$-Lie algebras do not generalize to Leibniz $n$-algebras. For example, the Leibniz $n$-kernel can coincide with the $n$-algebra which is possible for $n=2$ if and only if the Leibniz algebra is abelian. Also, there are examples of Leibniz $n$-algebras with invertible right multiplication operators and Cartan subalgebras of different dimensions \cite{cartan}, which are not in accordance with the corresponding results in Lie, Leibniz and $n$-Lie algebras. 

The motivation of this paper is to continue developing theory on Leibniz $n$-algebras and find analogs and variations in relation to results established in $n$-Lie algebras. One such direction of study is through a ``forgetful'' functor $\U_n$ from the category of Leibniz algebras $\Lb$ to Leibniz $n$-algebras ${}_n\Lb$. Given a Leibniz algebra $\mathfrak{L}$, authors of \cite{CLP} introduce a Leibniz $n$-algebra structure on the same vector space via $[x_1,\dots,x_n]=[x_1,[\dots, [x_{n-2},[x_{n-1},x_n]]\dots]]$. We prove that $\U_n(\Lb)$ is not a full subcategory of the category ${}_n\Lb$ and our goal is to investigate the structure of Leibniz $n$-algebras in the category $\U_n(\Lb)$.

Section 1 is an introduction to Leibniz algebras, Leibniz $n$-algebras and the forgetful functor $\U_n(\mathfrak{L})$.  Using the recent results of \cite{bipartite} we describe all ideals of a semisimple Leibniz algebra $\mathfrak{L}$ and compare them with the ideals of $\U_n(\mathfrak{L})$ in Section \ref{main}. 

In Section \ref{kernel}, we discuss the Leibniz $n$-kernel of a Leibniz $n$-algebra and show that for a Leibniz $3$-algebra from $\U_3(\Lb)$, the subspace spanned by the ``square'' of the elements, $\Span\{[x,x,y] \mid x,y\in \mathfrak{L} \}$ is the Leibniz $3$-kernel. For a Leibniz $n$-algebra with an invertible right multiplication operator it is known by \cite{cartan} that the Leibniz $n$-kernel is the whole $n$-algebra. We present a short proof of this fact and show that the converse does not work for $\U_n(\mathfrak{L})$, where $\mathfrak{L}$ is any semisimple Leibniz algebra. 

Section \ref{main} introduces the notions of simplicity and semisimplicity for a Leibniz $n$-algebra in a similar way as in $n$-Lie algebras. For a non-Lie Leibniz algebras $(n=2)$, both simplicity and semisimplicity are defined differently, taking into account the Leibniz kernel, which is a nontrivial ideal.  We establish that a Leibniz $n$-algebra $\U_n(\mathfrak{L})$ is simple, if and only if $\mathfrak{L}$ is a simple Lie algebra. Furthermore, we prove an analogue of Levi decomposition and establish that a Leibniz $n$-algebra $\U_n(\mathfrak{L})$ is semisimple if and only if $\mathfrak{L}$ is a semisimple Lie algebra. 

All vector spaces, modules, algebras and $n$-algebras in this work are finite-dimensional over a field $\mathbb{K}$ of characteristic zero.  

\section{Preliminaries} 
\subsection{Leibniz algebras}

A Leibniz algebra $\mathfrak{L}$ is a vector space over a field $\mathbb{K}$ with a bilinear bracket $[-,-]: \mathfrak{L} \times \mathfrak{L} \to \mathfrak{L}$ that satisfies the \textit{Leibniz identity}:
\[[[x,y],z]=[[x,z],y]+[x,[y,z]].\]
If the Leibniz algebra $\mathfrak{L}$ is skew-symmetric i.e. $[x,x] = 0$ for all $x \in \mathfrak{L}$, then the Leibniz identity becomes the Jacobi identity. A linear space spanned by the squares of the elements of a Leibniz algebra $\mathfrak{L}$ constitutes a two-sided ideal called the \textit{Leibniz kernel} and is denoted by $\Leib(\mathfrak{L})$. The quotient algebra $\mathfrak{L}/\Leib(\mathfrak{L})$ is a Lie algebra, called the \textit{liezation} of $\mathfrak{L}$. The Leibniz kernel is trivial if and only if the Leibniz algebra is abelian (that is $[\mathfrak{L},\mathfrak{L}]=\{0\}$) and $\Leib(\mathfrak{L})=\mathfrak{L}$  if and only if $\mathfrak{L}$ is a Lie algebra. Therefore, a non-Lie Leibniz algebra always admits a nontrivial ideal - the Leibniz kernel and we use this definition of simplicity for a non-abelian Leibniz algebra introduced in \cite{Dzumadil'daev}.

\begin{definition}\label{simpleLeibniz}
	A Leibniz algebra $\mathfrak{L}$ is called simple if the only ideals are $\{0\}, \Leib(\mathfrak{L})$ and $\mathfrak{L}$.
\end{definition}

\begin{definition}
	A subalgebra $K$ of $\mathfrak{L}$ is solvable if  there exists $m \in \mathbb{N}$ such that $K^{(m)}=\{0\}$, where $K^{(n)}=[K^{(n-1)},K^{(n-1)}]$ for any positive integer $n$ and $K^{(0)} = K$.
\end{definition}

As in Lie algebra theory, the sum of solvable ideals of a Leibniz algebra is a  solvable ideal. 
\begin{definition}
	The maximal solvable ideal of a Leibniz algebra $\mathfrak{L}$ is called the radical of  $\mathfrak{L}$, denoted $\rad(\mathfrak{L})$. 
\end{definition}

As in the definition of simple Leibniz algebras, the notion of semisimplicity of a Leibniz algebra is different from semisimplicity of Lie algebras. 
\begin{definition}
	A Leibniz algebra $\mathfrak{L}$ is called semisimple if $\rad(\mathfrak{L})=\Leib(\mathfrak{L})$.
\end{definition}

Note that a simple Leibniz algebra is also semisimple. Semisimple Leibniz algebras satisfy $[\mathfrak{L},\mathfrak{L}]=\mathfrak{L}$, i.e. they are perfect. There is an analog of Levi's theorem that describes the structure of Leibniz algebra. 

\begin{theorem}\cite{Barnes}
	Let $\mathfrak{L}$ be a Leibniz algebra. Then $\mathfrak{L} = \mathfrak{g} \ltimes \rad(\mathfrak{L})$, where $\mathfrak{g}$ is a semisimple Lie subalgebra of $\mathfrak{L}$.
\end{theorem}

For a semisimple Leibniz algebra $\mathfrak{L}$, the semisimple Lie subalgebra $\mathfrak{g}$ is the liezation $\mathfrak{g}_L$ of $\mathfrak{L}$. Since $\Leib(\mathfrak{L})$ is a $\mathfrak{g}_L$-module, by Weyl's semisimplicity $\Leib(\mathfrak{L})$ decomposes into a direct sum $\oplus_{k=1}^n I_k$ of simple $\mathfrak{g}_L$-submodules. Moreover, $\mathfrak{g}_L$ is a direct sum $\oplus_{i=1}^m\mathfrak{g}_i$ of simple Lie subalgebras and we have the following. 

\begin{corollary}\label{levi_semisimple} Let $\mathfrak{L}$ be a semisimple Leibniz algebra. Then 
	\begin{equation*}
		\mathfrak{L} = \mathfrak{g}_L \ltimes \Leib(\mathfrak{L})=(\oplus_{i=1}^m\mathfrak{g}_i)\ltimes (\oplus_{k=1}^n I_k).
	\end{equation*}
\end{corollary}
By \cite[Theorem 3.5, 3.8]{bipartite} the structure of an indecomposable semisimple Leibniz algebra $\mathfrak{L}=(\oplus_{i=1}^m \mathfrak{g}_i) \ltimes (\oplus_{k=1}^n I_k)$ can be determined by a connected bipartite graph with bipartition $(\{ \mathfrak{g}_1,\dots, \mathfrak{g}_m\}, \{ I_1, \dots, I_n \})$ and edges between $\mathfrak{g}_i$s and $I_j$s if $[I_j, \mathfrak{g}_i]=I_j$.

In the next theorem we describe all ideals of a semisimple Leibniz algebra. Finding the ideals of indecomposable semisimple Leibniz algebras suffice.  Given  a graph  $G=(V,E)$ and  a subset $A$ of $V$ we denote by $N(A)$ the set of all vertices in $G$ that are adjacent to at least one vertex in $A$.

\begin{theorem}
	The ideals of an indecomposable semisimple Leibniz algebra $\mathfrak{L}=(\oplus_{i=1}^m \mathfrak{g}_i) \ltimes (\oplus_{k=1}^n I_k)$ are the following:
	\begin{enumerate}
		\item[(i)] Abelian ideals $\oplus_{k\in A}I_k$, for any subset $A\subseteq \{1,\dots, n\}$;
		\item[(ii)] Non-solvable ideals $(\oplus_{i\in B} \mathfrak{g}_i)\ltimes (\oplus_{j\in N(B)} I_j) \oplus (\oplus_{k\in C} I_k)$, where $B\subseteq \{1,\dots, m\}$ and $C\subseteq \{1,\dots, n\}\setminus N(B)$.
	\end{enumerate}
\end{theorem}

\begin{proof} Let $X$ be an ideal of $\mathfrak{L}$. Since $\mathfrak{L}$ is a $\oplus_{i=1}^m \mathfrak{g}_i$-module we obtain that $X$ is a $\oplus_{i=1}^m \mathfrak{g}_i$-submodule of $\mathfrak{L}$.
	
	Assume $X\subseteq \oplus_{k=1}^n I_k$. By Weyl's semisimplicity $X$ is a Lie algebra module is a direct sum of simple $\oplus_{i=1}^m \mathfrak{g}_i$-submodules. Since $I_k$ for all $1\leq k \leq n$ are simple $\oplus_{i=1}^m \mathfrak{g}_i$-modules, $X$ is direct sum of some $I_k$'s and we obtain (i). 
	
	Now assume that $X\not \subseteq \oplus_{k=1}^n I_k$. By the similar argument we obtain that $X$ as a Lie algebra module is a direct sum of some $\mathfrak{g}_i$'s and $I_k$'s. Let $B$ be a subset of $\{1,\dots, m\}$ such that $\mathfrak{g}_i\in X$ only for  $i\in B$. For any $I_k\in N(\mathfrak{g}_i)$ by the construction of semisimple Leibniz algebra we have $I_k=[I_k,\mathfrak{g}_i]\subseteq [I_k,X]\subseteq X$ and therefore, $\oplus_{k\in N(B)} I_k\subseteq X$. Since $[I_k, \mathfrak{g}_i]=0$ whenever $I_k \notin N(\mathfrak{g}_i)$, one can add direct summands $I_k$'s to the ideal $X$  such that $k\in \{1, \dots, n\} \setminus N(B)$. 	
\end{proof}

Obviously, ideals of a semisimple Leibniz algebra are all possible summands of the direct sum of the ideals of each indecomposable subideals of the Leibniz algebra. 

\subsection{Leibniz $n$-algebras} 
A Leibniz $n$-algebra \cite{CLP} is defined as a $\mathbb{K}$-module $L$ equipped with a linear $n$-ary operation $[-,\ldots,-]: L^{\otimes n}\longrightarrow L$ satisfying the identity: 
\[[[x_1,\ldots,x_n],y_1,\ldots,y_{n-1}]=\sum^{n}_{i=1}[x_1,\ldots,x_{i-1},[x_i,y_1,\ldots,y_{n-1}],x_{i+1},\ldots,x_n].\]
Note that for $n=2$ this is the definition of Leibniz algebra. Moreover, if the bracket $[-,\dots,-]$ factors through $\Lambda^nL$, then $L$ is an  $n$-Lie algebra introduced in \cite{Fil}.

Since  $n$-ary multiplication of Leibniz $n$-algebras is not necessarily skew-symmetric, in \cite{cartan} some variation of the definition of ideal is given. 

\begin{definition} A subspace $I$ of a Leibniz $n$-algebra $L$ is called an $s$-ideal of $L$, if $$[\underbrace{L,\dots,L}_{s-1},I,L,\dots,L]\subseteq I.$$ If $I$ is an $s$-ideal for all $1\leq s \leq n$, then $I$ is called an ideal.
\end{definition}

Consider the $n$-sided ideal
$$\Leib_n(L) =  \text{ideal} \langle
[x_1,\dots,x_i, \dots , x_j, \dots , x_n]   \mid  \exists i, j: x_i=x_j , \ x_1,\dots,x_n\in L\rangle ,$$
which is called the \textit{Leibniz n-kernel} of $L$. The quotient $n$-algebra $L/\Leib_n(L)$ is clearly an $n$-Lie algebra. 

\begin{definition} A linear map $d$ defined on a Leibniz $n$-algebra $L$ is called a derivation  if
	\[
	d([x_1,x_2,\dots , x_n])=\sum_{i=1}^n \ [x_1,\dots d(x_i),\dots , x_n].
	\] 
\end{definition}

The space of all derivations of a given Leibniz $n$-algebra $L$ is denoted by $\Der(L)$ and forms a Lie algebra with respect to the commutator \cite{cartan}.  
Given an arbitrary element $x=(x_2,\dots,x_n) \in L^{\otimes(n-1)}$ consider the operator $R[x]: L \to L$ of right multiplication defined for all $z\in L$ by $$R[x](z)=[z,x_2,\dots,x_n].$$

As in $n$-Lie algebras, a right multiplication operator is a derivation and the space $R[L]$ of all right multiplication operators forms a Lie algebra ideal of $\Der(L)$ \cite{cartan}. Given a Leibniz $n$-algebra $L$, one can associate the Lie algebras $R[L]$ or $\Der(L)$ to $L$. The following statement from \cite[Propisition 3.2]{CLP} shows how a Leibniz $n$-algebra can be constructed from a given Leibniz algebra. 

\begin{proposition}	 Let $\mathfrak{L}$ be a Leibniz algebra with the product $[ - , - ]$. Then the vector space $\mathfrak{L}$ can be equipped with the Leibniz $n$-algebra structure with the following product:
	$$ [x_1, x_2, \dots, x_n]:= [x_1,[\dots,[x_{n-2}, [x_{n-1},x_n]]\dots]].$$
\end{proposition}
It follows that the authors of \cite{CLP} has built a ``forgetful" functor $\U_n: \Lb \to {}_{n}\Lb$ from the category of Leibniz algebras to the category of Leibniz $n$-algebras. The restriction of $\U_n$ on $\Lie$ does not necessarily fall into the category of $n$-Lie algebras ${}_{n}\Lie$. Conversely, given an $n$-Lie algebra $L$, there is a construction by  Daletskii and Takhtajan  \cite{DaTa} of a Leibniz algebra on $L^{\otimes (n-1)}$ with the bracket
\[
[l_1\otimes \dots \otimes l_{n-1}, l_1'\otimes \dots \otimes l_{n-1}']= \sum_{1\leq i \leq n-1} l_1\otimes \dots \otimes [l_i,l_1',\dots, l'_{n-1}]\otimes \dots \otimes l_{n-1}.
\]
This Leibniz algebra is called the \textit{basic Leibniz algebra} of an $n$-Lie algebra $L$ in \cite{DaTa} and comparisons with other Lie algebra constructions from a given $n$-Lie algebra are studied in \cite{n-Lie}. It is straightforward that Daletskii-Takhtajan's construction  gives rise to a functor $\mathcal{D}_{n-1}: {}_{n}\Lb \to \Lb$. Remarkably, $\U_n$  is not a right adjoint of $\mathcal{D}_{n-1}$. Indeed, for abelian Leibniz algebra $\mathfrak{L}$ and abelian $n$-algebra $L$ one can check that 
\begin{align*}
	\Hom_{\Lb}(\mathcal{D}_{n-1}(L), \mathfrak{L})&=\Hom_{\mathbb{K}}(L^{\otimes n-1}, \mathfrak{L}),\\
	\Hom_{{}_n\Lb}(L,\U_n(\mathfrak{L}))&= \Hom_{\mathbb{K}}(L,\mathfrak{L}),
\end{align*}	
and they are of different dimensions.

Note that we can define a Leibniz $n$-algebra as an $n$-algebra where every right multiplication operator is a derivation. For Lie and $n$-Lie algebras, an operator of right multiplication is singular. For Leibniz algebras (that is $n=2$) it is also true \cite{OmirovAlb}. However, for some Leibniz $n$-algebra ($n\geq 3$) an invertible operator of right multiplication exists, as shown in  \cite[Example 2.3]{cartan}. 

\begin{example} An $n$-ary algebra over a field $\mathbb{K}$ with the basis $\{e_1,\dots, e_m\}$ with the following multiplication
	$$[e_i,e_1,\dots , e_{n-1}]=\alpha_i e_i, \ \ \alpha\in \mathbb{K}$$
	where $\alpha_i \neq 0$ for all $1 \leq i \leq m, \ 
	\sum_{i=1}^{n-1}\alpha_i=0$ and all other products are zero is a Leibniz $n$-algebra. In this $n$-algebra the operator $R[e]$ is invertible, where $e=(e_1,\dots, e_{n-1})$.
\end{example}

Let $H$ be an ideal of a Leibniz $n$-algebra $L$. Put
$H^{(1)_k}=H$ and
\[
H^{(m+1)_k}
=\sum_{i_1+\dots+i_k=0}^{n-k}
[\underbrace{L,\dots,L}_{i_1},H^{(m)_k},\underbrace{L,\dots,L}_{i_2},H^{(m)_k},\dots,
\underbrace{L,\dots,L}_{i_k},H^{(m)_k},L,\dots,L]
\]
for all $1\leq k \leq n$ and $m\geq 1$.

\begin{definition}\cite{Frattini} An ideal $H$ of a Leibniz $n$-algebra $L$ is said to be $k$-solvable with index of $k$-solvability equal to $m$ if there exists $ m\in \mathbb{N}$ such that $H^{(m)_k}=0$ and $H^{(m-1)_k}\neq 0$. An ideal $H$ is called solvable if it is $n$-solvable. When $L = H$, $L$ is called  a $k$-solvable (correspondingly, solvable) Leibniz $n$-algebra.
\end{definition}
Notice that this definition agrees with the definition of
$k$-solvability of $n$-Lie algebras given in \cite{Kasymov}. Next statement is a straightforward generalization of the similar statement from \cite{Ling}.

\begin{proposition}\label{k<k' solvable} 
	A $k$-solvable ideal $H$ of a Leibniz $n$-algebra is $k'$-solvable for all $k<k'$. An ideal is called solvable if it is $n$-solvable. 
\end{proposition}

\begin{definition} 
	The maximal $k$-solvable ideal in Leibniz $n$-algebra $L$ is called the $k$-solvable radical of $L$, denoted $\Rad_k(L)$. $\Rad_k(L)$ is called the radical of $L$ if $k=n$ and is denoted $\Rad(L)$.
\end{definition}
In \cite[Theorem 4.5]{Frattini} the invariance of a $k$-radical under a derivation of a Leibniz $n$-algebra is proven. 

\begin{theorem} Let $D$ be a derivation of a Leibniz $n$-algebra $L$. Then $D(\Rad_k(L))\subseteq \Rad_k(L)$.
\end{theorem}

\section{Leibniz $n$-kernel of a Leibniz $n$-algebra}\label{kernel}
Consider a Leibniz $n$-algebra $L$ and let us introduce for all $1\leq i <j \leq n$ the following linear subspaces of $\Leib_n(L)$: 
$$I_{ij}=\Span \{[x_1,\dots,x_i,\dots,x_j,\dots, x_n] \mid x_i=x_j, \ x_1,\dots, x_n \in L \}.$$
Note that for $n=2$, $I_{12}$ is equal to the Leibniz kernel. In the general case, $\Leib_n(L)$ is generated by $ \sum_{1\leq i <j \leq n} I_{ij}$ and we have the following statement.

\begin{proposition}\label{I_ij_is_1-ideal}
	The linear space $I_{ij}$ is a 1-ideal of the Leibniz $n$-algebra $L$ for all $1\leq i < j \leq n$.
\end{proposition}
\begin{proof}
	From  
	\begin{align*}	
		[x_1,\dots,x_i + x_j,\dots,x_i+ x_j,\dots, x_n]&-[x_1,\dots,x_i,\dots,x_i,\dots, x_n]-[x_1,\dots,x_j,\dots,x_j,\dots, x_n]\\
		&= [x_1,\dots,x_i,\dots,x_j,\dots, x_n]  + [x_1,\dots,x_j,\dots,x_i,\dots, x_n]
	\end{align*} one obtains
	\begin{equation}\label{condition2}
	[x_1,\dots,x_i,\dots,x_j,\dots, x_n] + [x_1,\dots,x_j,\dots,x_i,\dots, x_n]\in I_{ij}.
	\end{equation}
	Consider the Leibniz $n$-identity 
	\begin{multline*}
		[[x_1,\dots,x_i,\dots,x_j,\dots, x_n],y_2,\dots, y_n]=  \sum_{k\neq i, j} [x_1,\dots, [x_k,y_2,\dots, y_n], \dots, x_n]\\
		 + [x_1,\dots, [x_i,y_2,\dots, y_n],\dots, x_j,\dots, x_n] +[x_1,\dots, x_i, \dots, [x_j,y_2,\dots, y_n],\dots, x_n].
	\end{multline*}
	If $x_i=x_j$, then the last two terms of the RHS by the property (\ref{condition2})	belong to $I_{ij}$ and every summand in the sum is in $I_{ij}$. Hence, $I_{ij}$ is a $1-$ideal of $L$ for all $1\leq i <j\leq n$. 
\end{proof}

\subsection{Results on $L=\U_3(\mathfrak{L})$ }
Since for Leibniz algebras $[a,[b,b]]=0$, which over the base field of characteristic not equal to 2 is equivalent to $[a,[b,c]]=-[a,[c,b]]$, we obtain 
\begin{equation}\label{condition}
[x,y,y]=0 \text{ and } [x,y,z]=-[x,z,y].
\end{equation}		
Hence, the ternary bracket in this case is a linear map $\mathfrak{L}\otimes (\mathfrak{L}\wedge \mathfrak{L})\to \mathfrak{L}$. 

An ideal $I$ of a Leibniz 3-algebra $L$ in the language of the underlying Leibniz algebra $\mathfrak{L}$ is defined by the following inclusions:
$$[I,[\mathfrak{L},\mathfrak{L}]]\subseteq I, \ [\mathfrak{L},[I,\mathfrak{L}]]\subseteq I, \ [\mathfrak{L},[\mathfrak{L},I]]\subseteq I. $$
Note that, since $il+li\in \Leib(\mathfrak{L})$ for any $i\in I, \ l\ \in \mathfrak{L}$ and $[\mathfrak{L},\Leib(\mathfrak{L})]=0$ we have  $[\mathfrak{L},[I,\mathfrak{L}]]=[\mathfrak{L},[\mathfrak{L},I]]$. Therefore, we have the following statement. 

\begin{proposition} 
	A subspace $V$ of  a Leibniz 3-algebra $\U_3(\mathfrak{L})$ is an ideal if and only if \[
	[V,[\mathfrak{L},\mathfrak{L}]]\subseteq V \textrm{ and } [\mathfrak{L},[V,\mathfrak{L}]]\subseteq V.
	\]
\end{proposition} 
\begin{proposition} The subspace $I_{12}$ of $\U_3(\mathfrak{L})$ is an ideal and is equal to $\Leib_3(L)$.
	
\end{proposition}
\begin{proof} Using equalities (\ref{condition}) we have $I_{23}=\{0\}$ and $I_{13}=I_{12}$. Therefore, the ideal $\Leib_3(L)$ is generated by $I_{12}$. By Proposition \ref{I_ij_is_1-ideal} $I_{12}$ is a 1-ideal. Moreover, one can establish 
	\begin{equation*}
		[x_1,x_2,x_3]-(-1)^{sgn(\pi)}[x_{\pi(1)},x_{\pi(2)}, x_{\pi(3)}] \in I_{12}
	\end{equation*}
	for all permutations $\pi \in S_3$, which implies that $I_{12}$ is an ideal and thus, is equal to $\Leib_3(L)$. 
\end{proof}

Since a Leibniz $3$-algebra is a $3$-Lie algebra if and only if $\Leib_3(L)=\{0\}$, we have immediately the following statement. 

\begin{corollary} $\U_3(\mathfrak{L})$ is a $3$-Lie algebra if and only if the Leibniz algebra $\mathfrak{L}$ satisfies $[x,[x,y]]=0$ for all $x,y\in \mathfrak{L}$.
\end{corollary}

\subsection{Right multiplication operator of $\U_n(\mathfrak{L})$ and the Leibniz $n$-kernel}  
Recall, that for a Lie and $n$-Lie algebras, all adjoint (right multiplication) maps are singular due to skew-symmetry. For Leibniz algebras, all right multiplications are also singular \cite{OmirovAlb}. However, for Leibniz $n$-algebras, as shown in \cite[Lemma 2.3]{cartan} right multiplication operators might be invertible. In this subsection we give a short proof of this fact generalizing the proof of \cite[Proposition 5]{Gor} and prove that the converse does not hold.

\begin{proposition} \label{non_degenrate_then_I=L} Let $L$ be a finite-dimensional Leibniz $n$-algebra over a field of characteristic zero. If it admits an invertible operator of right multiplication then $\Leib_n(L)=L$. 
\end{proposition}

\begin{proof}  Let there be an invertible right multiplication operator $R[a]:L\to L$, where $a=(a_2,\dots, a_n)$. Assume that $\Leib_n(L) \neq L$. If $\Leib_n(L)=\{0\}$, then $L$ is an $n$-Lie algebra which contradicts to invertibility of $R[a]$. Hence, $\Leib_n(L)$ is a non-trivial ideal of $L$ and note that it is an invariant subspace of $R[a]$. The right multiplication map $R[a]$ induces an invertible right multiplication operator in the quotient $n$-Lie algebra  $L/\Leib_n(L)$, which is a contradiction. Therefore, $\Leib_n(L)=L$.
\end{proof}

The converse statement is false, and an example of a Leibniz $n$-algebra with $\Leib_n(L)=L$ and all right multiplication maps being singular is constructed in \cite[Example 2.4]{cartan}, which is a $\U_n$ of a simple Leibniz algebra with $\mathfrak{sl}_2$ as its liezation. Obviously, any right multiplication operator for any Leibniz $n$-algebra in $\U_n(\Lb)$ is singular due to a right multiplication being singular in Leibniz algebras (\cite{OmirovAlb}, \cite{Gor}). Below we prove that any Leibniz $n$-algebra $\U_n(\mathfrak{L})$ for a semisimple Leibniz algebra $L$ serves as an example for the converse of Proposition \ref{non_degenrate_then_I=L} to be false. 

\begin{theorem}\label{Leib_n(L)=L}
	For any semisimple Leibniz algebra $\mathfrak{L}$, the Leibniz $n$-kernel of $\U_n(\mathfrak{L})$ coincides with $\U_n(\mathfrak{L})$.
\end{theorem}
\begin{proof} Denote by $L=\U_n(\mathfrak{L})$ and let us apply the decomposition of semisimple Leibniz algebra from Corollary \ref{levi_semisimple}. From Lie theory it is well-known that every simple Lie algebra contains a subalgebra isomorphic to $\mathfrak{sl}_2$. From $[h,e]=2e$ we obtain $2^{n-1}e=[h,h,\dots,h,e]$ which belongs to $\Leib_n(L)$. Hence, $\Leib_n(L)$ contains some elements from each simple Lie algebra $\mathfrak{g}_i$ in the decomposition of $\mathfrak{g}_L=\oplus_{i=1}^m \mathfrak{g}_i$. Since $\Leib_n(L)$ is a $\mathfrak{g}_L$-module, in fact $\Leib_n(L)$ contains every $\mathfrak{g}_i$ for all $1\leq i \leq n$ due to simplicity of the Lie subalgebras.
	
	Next, from $[\Leib(\mathfrak{L}), \mathfrak{g}_L]=\Leib(\mathfrak{L})$ and since $\mathfrak{g}_L$ is perfect, we have 
	\[
	\Leib(\mathfrak{L})=[\Leib(\mathfrak{L}),\mathfrak{g}_L,\dots, \mathfrak{g}_L]\subseteq [\Leib(\mathfrak{L}), \Leib_n(L), \dots,  \Leib_n(L)]\subseteq  \Leib_n(L).
	\]
	Thus $\Leib_n(L)\supseteq \mathfrak{g}_L\oplus \Leib(\mathfrak{L})=L$ which completes the proof. 
\end{proof}

\section{Structure of Leibniz $n$-algebra $\U_n(\mathfrak{L})$ }\label{main}
It is straightforward that a subalgebra $\mathfrak{A}$ of a Leibniz algebra $\mathfrak{L}$ constitutes a Leibniz $n$-subalgebra $\U_n(\mathfrak{A})$ of $\U_n(\mathfrak{L})$. Moreover, we have the following statement.
\begin{proposition}\label{directsum_CLP}  $\U_n(\mathfrak{L}_1\oplus \mathfrak{L}_2)=\U_n(\mathfrak{L}_1)\oplus \U_n(\mathfrak{L}_2)$ for any Leibniz algebras $\mathfrak{L}_1$ and $\mathfrak{L}_2$.
\end{proposition}	
\begin{proof} 	Let $(x_1 , y_1), (x_2 , y_2), \dots , (x_n , y_n) \in \mathfrak{L}_1 \oplus \mathfrak{L}_2$. Using the product in $\U_n(\mathfrak{L}_1\oplus \mathfrak{L}_2)$ we have
	\begin{align*}
		[(x_1 , y_1), (x_2 , y_2), \dots, (x_n , y_n)]&=[(x_1 , y_1),[\dots,[(x_{n-2} , y_{n-2}) , [(x_{n-1} , y_{n-1}),(x_n , y_n)]]\dots]]\\
		=& [(x_1 , y_1),[\dots,[ (x_{n-2} , y_{n-2}) , \left([x_{n-1}, x_n] , [y_{n-1}, y_n]\right) ]\dots]]\\
		= & [(x_1 , y_1),[\dots,[ (x_{n-3} , y_{n-3}) , \left([x_{n-2},[x_{n-1}, x_n]] , [y_{n-2},[y_{n-1}, y_n]]\right) ]\dots]]\\
		=& \left([x_1 [\dots,[x_{n-2} , [x_{n-1}, x_n]]\dots]] , [y_1 [\dots,[y_{n-2} , [y_{n-1}, y_n]]\dots]] \right)\\
		= & ([x_1 , x_2 ,\dots, x_n] , [y_1 , y_2 ,\dots, y_n])
	\end{align*}	
	which implies the result. 
\end{proof}	

Next, we introduce of the notion of simplicity as in \cite{Fil}. Note the difference with Definition \ref{simpleLeibniz} for Leibniz algebras. 

\begin{definition}
	A non-abelian Leibniz $n$-algebra $(n\geq 3)$ is called simple if the only ideals are $\{0\}$ and the $n$-algebra itself. 
\end{definition}
Recall from \cite{Fil} an important example of an $(n+1)$-dimensional $n$-Lie algebra which is an analogue of the three-dimensional Lie algebra with the cross product as multiplication. 
\begin{example}
	Let $\mathbb{K}$ be a field and $V_n$ an $(n+1)$-dimensional $\mathbb{K}$-vector space with a basis $\{ e_1,\dots, e_{n+1}\}$. Then $V_n$, equipped with the skew-symmetric $n$-ary multiplication induced by
	\[
	[e_1,\dots, e_{i-1},e_{i+1},\dots, e_{n+1}]=(-1) ^{n+1+i}e_i,\quad 1\leq i \leq n+1,
	\]
	is an $n$-Lie algebra.
\end{example}
\noindent This $n$-Lie algebra is a simple $n$-Lie algebra.  Remarkably, as shown in~\cite{Ling}, 
over an algebraically closed field $\mathbb{K}$ all simple $n$-Lie algebras are isomorphic to $V_n$. 

\begin{proposition}\label{simple_nLie_not_U_n}
	The simple  $n$-Lie algebra $V_n$ is not in the category $\U_n(\Lb)$.
\end{proposition}

\begin{proof} Assume that there exists a Leibniz algebra $\mathfrak{L}$, such that $V_n=\U_n(\mathfrak{L})$. Let $\{ e_1,\dots, e_{n+1}\}$ be a basis of $\mathfrak{L}$. Suppose 
	$[e_i,e_j]=\alpha_{ij} ^1 e_1$ + $\alpha_{ij} ^2 e_2$ + \dots + $\alpha_{ij} ^{n+1} e_{n+1}$ for some $\alpha_{ij}^k\in \mathbb{K}$.

	Let us denote by $E$ the right ordered product $[e_{k_1},[\dots,[e_{k-3},[e_{k_{n-2}},e_{k_{n-1}}]]\dots]]$ of basis elements such that $k_1,\dots,k_{n-1} \in \{1,\dots,n+1\}$ and $k_i \neq k_j$ for i $\neq$ j. Fix an integer $i\in \{1,\dots, n+1\}$ and let some $k_j=i$. On one hand, \[
	[[e_i,e_i],e_{k_1},e_{k_2},\dots,e_{k_{n-2}},e_{k_{n-1}}]=[[e_i,e_i],E]=[[e_i,E],e_i]+[e_i,[e_i,E]]=0,
	\] 
	since $[e_i,E]=[e_i,e_{k_1},e_{k_2},\dots,e_{k_{n-2}},e_{k_{n-1}}]=0$.	On the other hand, 	
	\begin{align*}
		[[e_i,e_i],e_{k_1},e_{k_2},& \dots,e_{k_{n-2}},e_{k_{n-1}}]= \sum_{r=1}^{n+1}\alpha_{ii}^r[e_r,e_{k_1},e_{k_2},\dots,e_{k_{n-2}},e_{k_{n-1}}]\\
		=& \alpha_{ii}^p[e_p,e_{k_1},e_{k_2},\dots,e_{k_{n-2}},e_{k_{n-1}}]+\alpha_{ii}^qr[e_q,e_{k_1},e_{k_2},\dots,e_{k_{n-2}},e_{k_{n-1}}]\\
		=& \pm \alpha_{ii}^p e_q\pm \alpha_{ii}^qe_p,
	\end{align*}
	where $\{p,q\}=\{1,\dots, n+1\}\setminus \{k_1,\dots,k_{n-1} \}$.  Hence, we obtain $\alpha_{ii} ^{p}=\alpha_{ii} ^{q}=0$ for all $\{p,q\}=\{1,\dots, n+1\}\setminus \{k_1,\dots,k_{n-1} \}$. Choosing all possible subsets of $n-2$ elements from the set $\{1,\dots, n+1\}\setminus \{i\}$ for $\{k_1,\dots, l_{n-1}\}\setminus \{i\}$, similarly we establish that $\alpha_{ii}^r=0$ for all $r\neq i$. From $0=[e_i,\dots,e_i]=[e_i[\dots,[e_i, [e_i,e_i]]\dots]]=(\alpha^i_{ii})^{n-1}e_i$
	we obtain $\alpha_{ii}^i=0$ and $[e_i,e_i]=0$.

	Now fix some $1\leq i\neq  j \leq n+1$ and choose different $k_1,\dots,k_{n-1}$ from $\{1,\dots,n+1\}$ so that $i$ and $j$ are also chosen. Similarly, using Leibniz identity and the $n$-Lie bracket one establishes $[[e_i,e_j],E]=0$. On the other hand, $\displaystyle [[e_i,e_j],E]=\sum_{r=1}^{n+1}\alpha^r_{ij}[e_r,E]$ and going through all possible options for $k_1,\dots,k_{n-1}$ yields $[e_i,e_j]=\alpha_{ij}^i e_i +\alpha_{ij}^j e_j$. From \[
	0=[e_i,\dots,e_i,e_j]=[e_i,[\dots,[e_i,[e_i,e_j]]\dots]]=\alpha_{ij}^i(\alpha_{ij}^j)^{n-2}e_i+(\alpha_{ij}^j)^{n-1}e_j
	\]
	we obtain $\alpha_{ij}^j=0$ and $[e_i,e_j]=\alpha_{ij}^ie_i$. This implies that $$e_{n+1}=[e_1,e_2,\dots,e_{n}]=[e_1,[\dots,[e_{n-2},[e_{n-1},e_n]]\dots]]\in \Span\{e_1\},$$ which is a contradiction. 
\end{proof}

\begin{proposition}\label{correspondence_ideals} Let $I$ be an ideal of a Leibniz algebra $\mathfrak{L}$. Then $I$ is an ideal of the Leibniz n-algebra $L=\U_n(\mathfrak{L})$. 
\end{proposition}

\begin{proof} Let  $s\in$ $\{1,2,\dots,n\}$ be given. Then we have $
[\underbrace{L,\dots,L}_{s-1},I,L,\dots,L]=$ 
\newline $[\underbrace{\mathfrak{L},[\dots,}_{s-1}[I,[\mathfrak{L},[\dots,[\mathfrak{L},[\mathfrak{L},\mathfrak{L}]]\dots]]]\dots]]
		\subseteq  [\underbrace{\mathfrak{L},[\dots,[\mathfrak{L}}_{s-1},[I,\mathfrak{L}]]\dots]] 
		\subseteq  [\underbrace{\mathfrak{L},[\dots,[\mathfrak{L},[\mathfrak{L},}_{s-1} I]]\dots]]$
	due to $[\mathfrak{L},[\dots ,[\mathfrak{L} , [\mathfrak{L},\mathfrak{L}]]\dots]] \subseteq \mathfrak{L}$.
	
	By definition we have  $[\mathfrak{L},I]\subseteq I$, and assuming   $[\underbrace{\mathfrak{L},[\dots,[\mathfrak{L},[\mathfrak{L}}_{k-1}, I]]\dots]] \subseteq I$ leads to 
	$[\underbrace{\mathfrak{L},[\dots,[\mathfrak{L},[\mathfrak{L},}_{k} I]]\dots]] = [\mathfrak{L},[\underbrace{\mathfrak{L},[\dots,[\mathfrak{L},[\mathfrak{L},}_{k-1} I]]\dots]]] \subseteq [\mathfrak{L}, I] \subseteq I$. Hence, by induction we obtain $[\underbrace{L,\dots,L}_{s-1},I,L,\dots,L] \subseteq I$ which concludes that $I$ is an ideal of $\U_n(\mathfrak{L})$.	
\end{proof}

The following example shows that the converse of Proposition \ref{correspondence_ideals} is not true. 

\begin{example}\label{more_ideals_in_CLP} Consider the two-dimensional Leibniz algebra $\mathfrak{L}_2$ with a basis $\{e, f\}$ such that 
	\[
	[e,e] = [e,f] = [f,e] = 0 \textrm{ and }[f,f] = e.
	\]
	$\U_n(\mathfrak{L}_2)$ is an abelian $n$-algebra and any subspace is an ideal of $\U_n(\mathfrak{L}_2)$. However, $\Span\{e\}$ is the only nontrivial ideal of the Leibniz algebra $\mathfrak{L}_2$.
\end{example}

\begin{remark} The category $\U_n(\Lb)$ is a not a full subcategory of ${}_n \Lb$. Indeed, consider Leibniz algebra $\mathfrak{L}_2$ from Example \ref{more_ideals_in_CLP} and note that any linear map from $\U_n(\mathfrak{L}_2)$ to $\U_n(\mathfrak{L}_2)$ is a Leibniz $n$-algebra homomorphism. However, a linear map $\phi:\mathfrak{L}_2\to \mathfrak{L}_2$ defined on the basis elements of $\mathfrak{L}_2$ by $\phi(e)=0, \phi(f)=f$ is not
	a Leibniz algebra homomorphism 	due to 
	$$0=\phi(e)=\phi([f,f]) \neq [\phi(f),\phi(f)]=[f,f]=e.$$
\end{remark}

The following statement describes all simple Leibniz $n$-algebras  in  $\U_n(\Lb)$.  

\begin{theorem}\label{simple_CLP}
	A Leibniz $n$-algebra $\U_n(\mathfrak{L})$ is simple if and only if $\mathfrak{L}$ is a simple Lie algebra. 
\end{theorem}

\begin{proof}  	
	Let $L=\U_n(\mathfrak{L})$ be a simple Leibniz $n$-algebra. Then by Proposition \ref{correspondence_ideals} the Leibniz algebra $\mathfrak{L}$ does not admit non-trivial ideals which implies that $\mathfrak{L}$ is a simple Lie algebra. Conversely,  let $\mathfrak{L}$ be a simple Lie algebra and consider $L=\U_n(\mathfrak{L})$. If $I$ is a nontrivial ideal of $L$, then $I \supseteq [I,L,\dots,L]=[I,[\mathfrak{L},[ \dots [\mathfrak{L},[\mathfrak{L},\mathfrak{L}]]\dots]]]=[I,\mathfrak{L}]$. Hence, $I$ is a nontrivial ideal of the simple Lie algebra $\mathfrak{L}$, which is a contradiction. 
\end{proof}

By Proposition \ref{simple_nLie_not_U_n}, simple Leibniz $n$-algebras in $\U_n(\Lb)$ are not $n$-Lie algebras. One may wonder, what happens if simplicity is defined for Leibniz $n$-algebras as in Leibniz algebras, leaving freedom to admit exactly one nontrivial ideal. The following proposition describes all Leibniz $n$-algebras in $\U_n(\Lb)$ that admit exactly one nontrivial ideal. However, that ideal is not the Leibniz $n$-kernel. 
\begin{proposition}\label{simpleLeibniz_CLP}
	A Leibniz $n$-algebra $\U_n(\mathfrak{L})$ admits only one nontrivial ideal if and only if $\mathfrak{L}$ is a simple non-Lie Leibniz algebra or a Lie algebra $\mathfrak{g}\ltimes V$, where $\mathfrak{g}$ is a simple Lie algebra and $V$ is a simple $\mathfrak{g}$-module. The nontrivial ideal of $\U_n(\mathfrak{L})$ is either $\Leib(\mathfrak{L})$ or $V$, correspondingly.
\end{proposition}
\begin{proof} 	Let $I$ be the only ideal of $\U_n(\mathfrak{L})$. By Proposition \ref{correspondence_ideals} it follows that the Leibniz algebra $\mathfrak{L}$ does not admit more than one nontrivial ideal. If $\mathfrak{L}$ is a non-Lie Leibniz algebra, since the Leibniz kernel $\Leib(\mathfrak{L})$ is a non-trivial ideal of $\mathfrak{L}$ it follows that $\mathfrak{L}$ is a simple Leibniz algebra and $I=\Leib(\mathfrak{L})$. If $\mathfrak{L}$ is a Lie algebra, by Theorem \ref{simple_CLP} it is not simple and Proposition \ref{correspondence_ideals} concludes that it admits exactly one ideal. Therefore, $\mathfrak{L}$ is a semi-direct product Lie algebra $\mathfrak{g}\ltimes V$ for some simple Lie algebra $\mathfrak{g}$ and a simple $\mathfrak{g}$-module $V$, and $I=V$.  
	
	Conversely, let $\mathfrak{L}$ be a simple non-Lie Leibniz algebra. By Corollary \ref{levi_semisimple} we have $\mathfrak{L}=\mathfrak{g}\ltimes \Leib(\mathfrak{L})$ and $\Leib(\mathfrak{L})$ is an irreducible $\mathfrak{g}$-module and $\mathfrak{g}$ is a simple Lie algebra. Note that $\mathfrak{L}$ as a $\mathfrak{g}$-module decomposes into the direct sum $\mathfrak{g}\oplus \Leib(\mathfrak{L})$ of simple $\mathfrak{g}$-modules. Let $I$ be an ideal of $\U_n(L)$. Then from  $I\supseteq [I,\mathfrak{g},\dots,\mathfrak{g}]=[I,[\mathfrak{g},[\dots, [\mathfrak{g},[\mathfrak{g},\mathfrak{g}]]\dots]]]=[I,\mathfrak{g}]$ we obtain that $I$ is a $\mathfrak{g}$-submodule of $\mathfrak{L}$. $I$ is a nontrivial module so it must either be $\Leib(\mathfrak{L})$ or $\mathfrak{g}$. However, the latter one is not an ideal of $\U_n(\mathfrak{L})$ due to $[\Leib(\mathfrak{L}), \mathfrak{g},\dots,\mathfrak{g}]=[\Leib(\mathfrak{L}),\mathfrak{g}]=\Leib(\mathfrak{L})$. Therefore, $I=\Leib(\mathfrak{L})$ is the only nontrivial ideal of $L$. Similarly, one obtains $I=V$ in case $\mathfrak{L}$ is a Lie algebra $\mathfrak{g}\ltimes V$.
\end{proof}

Note that, from Theorem \ref{Leib_n(L)=L} it follows that the Leibniz $n$-kernel of $\U_n(\mathfrak{L})$ of a Lie or Leibniz algebra $\mathfrak{L}$ from Proposition \ref{simpleLeibniz_CLP} is equal to the whole $n$-algebra. 

\begin{theorem} The ideals of $\U_n(\mathfrak{L})$ for a semisimple Leibniz algebra $\mathfrak{L}$ are exactly the ideals of $\mathfrak{L}$. 
\end{theorem}

\begin{proof} 	The case when $\mathfrak{L}$ is a Lie algebra is straightforward from Proposition \ref{directsum_CLP}. Also, it suffices to consider the case of an indecomposable semisimple Leibniz algebra. Let $I$ be an ideal of $\U_n(\mathfrak{L})$.  Recall that $[I,[\mathfrak{L},[\dots,[\mathfrak{L},[\mathfrak{L},\mathfrak{L}]]\dots]]]\subseteq I$ and $[\mathfrak{L},[I,[\dots,[\mathfrak{L},[\mathfrak{L},\mathfrak{L}]]\dots]]]\subseteq I$.  Since semisimple Leibniz algebras are perfect, then  
	\begin{equation}\label{ideal_condition}
	[I,\mathfrak{L}]\subseteq I \textrm{ and } [\mathfrak{L},[I,\mathfrak{L}]]\subseteq I.
	\end{equation}
	By Proposition \ref{correspondence_ideals} the ideals of $\mathfrak{L}$ are ideals of $\U_n(\mathfrak{L})$.
	
	Assume by contradiction that there exists an ideal $H$ of $\U_n(\mathfrak{L})$ that is not an ideal of $\mathfrak{L}$. Then by (\ref{ideal_condition}) we have $[H, \mathfrak{L}]\subseteq$ $H$ which implies $[H,g_L] \subseteq H$, where $g_L$ is a semisimple Lie algebra (liezation of $\mathfrak{L}$) and $H$ is a $\mathfrak{g}_L$-module. Using  the structure of semisimple Leibniz algebra  we conclude that $H=(\oplus_{i \in I} \ g_i)\oplus(\oplus_{j \in J} \ I_j)$ for some $I\subseteq \{1,\dots, m\}, J\subseteq \{1,\dots, n\}$.  Furthermore, 
	\[
	[H, g_L] = [\oplus_{i\in I} \mathfrak{g}_i , \oplus_{i=1}^m \mathfrak{g}_i] \oplus [\oplus_{j\in J} I_j , \oplus_{i=1}^m  g_i] = (\oplus_{i\in I} g_i) \oplus (\oplus_{j\in J} I_j) = H.
	\]
	Thus, $H \supseteq [\mathfrak{L},[H, \mathfrak{L}]]=[\mathfrak{L},[H, \mathfrak{g}_L]]=[\mathfrak{L},H]$  and $[H, \mathfrak{L}] \subseteq H$ implies $H$ is an ideal of $\mathfrak{L}$ which is a contradiction.  
\end{proof}

Following Ling \cite{Ling} we extend the following notions from $n$-Lie algebras to Leibniz $n$-algebras. 

\begin{definition} Leibniz $n$-algebra $L$ is called $k$-semisimple if $\Rad_k(L)=\{0\}$. If $L$ is $n$-semisimple it is called semisimple. 
\end{definition}

Let $\mathfrak{L}=\mathfrak{g}\ltimes \rad(\mathfrak{L})$ be Levi decomposition of the Leibniz algebra $\mathfrak{L}$. We have the following analogue of Levi decomposition for $n$-algebras from the category $\U_n(\Lb)$. 

\begin{theorem}  Let $\mathfrak{L}$ be a Leibniz algebra and $\mathfrak{g}\ltimes \rad(\mathfrak{L})$ be Levi decomposition. Let $\mathfrak{g}_i$ be simple Lie subalgebras from the decomposition $\mathfrak{g}=\oplus_{i=1}^m \mathfrak{g_i}$. Then in Leibniz $n$-algebra $\U_n(\mathfrak{L})$, the $k$-solvable radical coincides with $\rad(\mathfrak{L})$ for all $2\leq k \leq n$  and there is the following analogue of the Levi decomposition
	$$\U_n(\mathfrak{L})=(\oplus_{i=1}^n \ \U_n(\mathfrak{g_i}))+ \Rad(\U_n(\mathfrak{L})).$$
\end{theorem}

\begin{proof} 	Let us prove first that for an ideal $I$ of $L$, $I^{(m)_2} \subseteq I^{(m)}$. For $m=2$,
	\begin{align*}
		I^{(2)_{2}} = & [I,I,L,\dots] + [I,L,I,L,\dots] + [L,I,I,L,\dots] + \dots + [L,\dots,I,I] \\
		\subseteq & [I,[I,\mathfrak{L}]] + [I,[\mathfrak{L},I]] + [\mathfrak{L},[I,I]] \subseteq [I,I] = I^{(2)}.
	\end{align*} 
	Assume $I^{(m-1)_{2}} \subseteq I^{(m-1)}$, then $I^{(m)_{2}} \subseteq [I^{(m-1)_{2}},I^{(m-1)_{2}}] $ follows similarly and inductively we obtain $I^{(m)_{2}} \subseteq I^{(m)}$ for all positive integers $m$. 
	
	If $m$ is the index of solvability of  $\rad(\mathfrak{L})$, then $\rad(\mathfrak{L})^{(m)_{2}} \subseteq \rad(\mathfrak{L})^{(m)} = \{0\}$ and by Proposition \ref{k<k' solvable} it follows that $\rad(\mathfrak{L})$ is a $k$-solvable ideal of $L$ for all $2\leq k \leq n$. Hence, $\rad(\mathfrak{L})\subseteq \Rad_k(L)$ for all $2\leq k \leq n$. Let us prove the converse inclusion and that all radicals coincide. Since $L$ is a $\mathfrak{g}$-module and $\Rad_k(L)$ is an ideal of $L$ we obtain that $\Rad_k(L)$ is a $\mathfrak{g}$-submodule of $L$. Let $\oplus_{i=1}^n \mathfrak{g}_i$ be the decomposition of $\mathfrak{g}$ into a direct sum of simple Lie subalgebras. Note that if $\mathfrak{g}\cap \Rad_k(L)\neq \{0\}$ then $\mathfrak{g}_i \subseteq \Rad_k(L)$ for some $1\leq i \leq n$. However, since $[\mathfrak{g_i},\mathfrak{g_i},\dots, \mathfrak{g_i}]=\mathfrak{g_i}$ this is a contradiction with solvability of the radical of $L$. Thus $\Rad_k(L)=\rad(\mathfrak{L})$ for all $2\leq k \leq n$. 
	
	As a vector space, $L=\mathfrak{L}=\mathfrak{g}\oplus \rad(\mathfrak{L})=\U_n(\mathfrak{g})\oplus \Rad(L)$. Moreover, by Proposition \ref{directsum_CLP} we obtain that $\U_n(\mathfrak{g})$ is a direct sum of simple Leibniz $n$-algebras $\U_n(\mathfrak{g}_i)$s and we obtain the desired decomposition.
\end{proof}

For a semisimple $\U_n(\mathfrak{L})$ Leibniz $n$-algebra, by definition the solvable radical is zero and by the result above, it follows that $\U_n(\mathfrak{L})$ is a direct sum of simple Leibniz $n$-algebras. 
\begin{conjecture} 
	A semisimple Leibniz n-algebra ($n\geq 3$) is a direct sum of simple Leibniz n-algebras.
\end{conjecture}

For $n$-Lie algebras this is \cite[Proposition 2.7]{Ling}. The structure of semisimple Leibniz algebra presented in Preliminaries shows why we exclude $n=2$ in this conjecture.

\section*{Acknowledgments}
We would like to thank Xabier Garc\'ia-Mart\'inez for reviewing the early version of the manuscript and providing useful comments.

\end{document}